\documentclass[12pt,draft]{amsart}
\usepackage{amssymb}
\usepackage{amsmath}
\usepackage{amsfonts}
\usepackage{amsthm}
\usepackage[english]{babel}
\theoremstyle{plain}
\newtheorem{ex}{Example}[section]

\newcommand{\Z}{\ensuremath{\mathbb{Z}}}
\newcommand{\N}{\ensuremath{\mathbb{N}}}

\newcommand{\F}{\ensuremath{\mathcal{F}}}
\newcommand{\I}{\ensuremath{\mathcal{I}}}

\def\d{\delta}

\def\s{\sigma}

 \theoremstyle{plain}

\newtheorem{lemma}{Lemma}
\newtheorem{proposition}{Proposition}
\theoremstyle{definition}

\theoremstyle{remark}

\newtheorem{remark}{Remark}

\numberwithin{equation}{section}

\def\d{\delta}

\def\s{\sigma}

\def\Ker{\operatorname{Ker}}
\def\Ran{\operatorname{Ran}}

\begin{document}
 \title[Martingale-coboundary representation for random fields ]{Martingale-coboundary representation for a class of random fields   \\*[0.6cm]}
\author{Mikhail Gordin}
\address{POMI (Saint Petersburg Department of the Steklov  Institute of Mathematics)\\
 27 Fontanka emb.\\
Saint Petersburg 191023 \\
Russian Federation}
\email{gordin@pdmi.ras.ru}
\thanks{ This work was supported in part by the grant NS 638.2008.1.
A part of this work was done at the Erwin Schr\"odinger
International Institute of Mathematical Physics (Vienna) when the
author participated at the workshops "Amenability beyond groups,"
"Algebraic, geometric and probabilistic aspects of amenability"
(2007) and "Structural Probabilty" (2008). } \keywords{Random
field, martingale difference, coboundary} \subjclass{Primary:
60G60; Secondary: 60Fxx}
\date{}
\vspace{0.5cm}
\vspace{1cm}
\normalsize
\maketitle
\section{Introduction}\label{0}
 Martingale approximation is one of methods of proving limit the\-o\-rems for sta\-tionary random
 sequences. The method, in its simplest version, consists of representing the original random
 sequence as the sum of a martingale difference sequence and a coboundary sequence. In  this
 introduction we give a brief sketch of this approach. The aim of the present paper
 is to extend the martingale approximation method to a certain class of random fields. This is the topic
 of the next two sections of the paper.   \\
Let $\xi=(\xi_n)_{n \in \mathbb Z}$ be a stationary (in the strict sense)
random sequence. Under certain assumptions \cite{Go1969} it can be represented in the form
$$ \xi_n = \eta_n + \zeta_n,$$
where $\eta=(\eta_n)_{n \in \mathbb  Z}$ is a stationary sequence of {\em martingale differences}  (this means that
$E(\eta_n|\eta_{n-1}, \eta_{n-2}, \dots)=0$ for all $ n \in \mathbb  Z$) , and $ \zeta=(\zeta_n)_{n \in \mathbb  Z}$ is
a so-called {\em coboundary} (or {\em coboundary sequence}) which can be written as
  $\zeta_n=\theta_n-\theta_{n-1}, n \in \mathbb  Z,$
by means of a certain stationary sequence  $\theta= (\theta_n)_{n \in \mathbb  Z}.$
 It is assumed that the random sequences  $\xi, \eta, \theta$ in this
 representation are stationarily connected, that is the
 sequence $\bigl((\xi_n, \eta_n, \theta_n)\bigr)_{n \in \mathbb Z}$ of random vectors
 is stationary. Let us observe that, while study\-ing the asymptotic distributions
 of the sums $\sum_{k=0}^{n-1} \xi_k,$ $n \ge 1,$ in many cases one can neglect by the contribution of the
 sequence $\zeta$ into these sums and extend to $\xi=\eta + \zeta$
  limit theorems originally known for the martingale difference $\eta$ only (notice that the
  limit theory for martingale differences is well developed). To be negligible in this sense, the sequence
  $\zeta$ needs not be a coboundary: some conditions are known \cite{MaWo2000,PeUt2005} under
  which approximation of the sums of the sequence  $\xi$ by those of the martingale difference
  sequence $\eta$ is precise enough to conclude that some limit theorems are applicable to $\xi$ once
  they hold for $\eta$; nevertheless, the difference $\zeta=\xi-\eta$  may not  be  a coboundary
  under these conditions. However, we consider here more special situation when
  the negligible summand does have a form of a coboundary: it is this case which admits
  the most transparent description and analysis and seems to be more appropriate for an attempt to
  extend the mar\-tin\-gale approach to random fields. \\
Conditions of limit theorems which are proved by means of the mar\-tin\-gale approximation are usually formulated
in terms of a ceratin filtration. This filtration is defined on the basic probability space; it is assumed to be stationarily connected with the sequence
 $\xi$ (the latter means  that the filtration is the sequence of the past $\s-$fields of a certain  auxiliary
 stationary sequence stationarily connected with $\xi).$ \, In general, the martingale approximation is applicable
 even if  $\xi$ is not adapted to this filtration. However, the adapted case deserves a special attention  not only by pedagogical reasons. It is this situation when there are more satisfactory answers to some natural questions,
such as those about the applicability of the Central Limit Theorem (CLT) and about the variance of the limiting normal
distribution. In the adapted case the sequence $(\xi_n)$  can be thought of as a non-anti\-ci\-pa\-ting function of a Markov chain. A simple condition  in terms of the transition operator (solvability of the so-called Poisson equation) guaranties the desired representation to hold which implies the ap\-pli\-ca\-bi\-li\-ty of the CLT. There
is a simple formula expressing the variance of the limiting normal distribution in terms
of the solution of the Poisson equation (see \cite{GoLi1978,Ma1978} and this Sect.1). Notice
now that the "time reversal" \, in the stationary case does  not hurt
the validity of conclusions about convergence in probability or in distribution: such
assertions are valid or not simultane\-ous\-ly for both
 the original and the reversed sequences. Thus, applying to the adapted case the
 time reversal, one obtains a convenient setup where without loss
 of generality one can assume that all stationary sequences of interest are given rise by
 a certain probability preserving transformation (the latter should be non-invertible in
 nontrivial cases). The decreasing filtration mentioned above arises in this situation as the sequence
 of  $\s-$fields of preimages of measurable sets with respect to the degrees of the basic
 transformation. It is this setting which we have chosen as a framework for a discussion of
 multivariate  generalizations of the martingale approximation method. Notice that various definitions
 of multivariate  arrays of martingale differences are possible (see, for example,
 \cite{BaDo1979, CaWa1975}). Our as\-sump\-tions lead us with necessity to one of them (see Remark \ref{condind})
 which is tightly related to one of several definitions in \cite{CaWa1975}.
 In the present paper we did not discuss in detail  these diverse definitions (though this
 topic is slightly touched in Remark \ref{condind}) because such a discussion seems to be  more
 appropriate in the context of limit theorems which will be considered elsewhere.  \\
 In the rest of Section 1 we remind how in such a setting a simplest result on
 martingale approximation for a random sequence is for\-mu\-lated. In the next sections
 of the paper we turn to establishing  an analogous re\-pre\-sentation for random fields generated by
 a class of measure pre\-serv\-ing actions of the additive semigroup of integral $d-$dimensional vectors with
 non\-nega\-tive entries.

 Let $T$ be a measure preserving transformation of a probability space $(X, \F, P).$
 Stationary sequences we are going to consider are of the form $(f \circ T^n)_{n \ge 0} $,
 where $f$ is a real-valued measurable function on $X.$ Set for $f \in L_2=L_2(X, \F, P)$ \,
 $ Uf=f\circ T,$ and let $ U^*:L_2 \to L_2 $ be the conjugate of the operator $U$.
 The operators $U$ è $U^*$ are, respectively, an isometry and a coisometry in $L_2$.
 Both of them preserve values of constant functions and map nonnegative functions to nonnegative ones.
 Consider $U^*$ as a transition operator of a Markov chain  taking values
in $X$ and having $P$ as a stationary distribution. The current state of
the chain uniquely determines the previous one by means of the transformation $T$.
 Let $E^{\mathcal G}$ and $I$  denote the conditional expectation operator with respect to some $\s-$field
$\mathcal G \subset \mathcal F$ and the identity operator, correspondingly. The relations
$$ (U^*)^n U^n= I, \, U^n (U^*)^n = E^{\, T^{-n}\mathcal F}, \, n \ge 0, $$
 hold between the operators $U$ è $U^*.$

Let us now assume that, for some function  $f \in L_2,$ a function
$g \in L_2$ solves the  {\em Poisson equation}
\begin{equation} \label{poisson1}
 f =g - U^*g.
\end{equation}
 Then, setting $h_1=U^*g,$ we have
$$U^*f =h_1 - U^*h_1, $$
which implies
$$  E^{\, T^{-1}\mathcal F}(f -Uh_1+h_1)= UU^*(f -Uh_1+h_1)=U(U^*f - h_1+U^*h_1)=0.$$
Since
$$f -Uh_1+h_1=g - U^*g - UU^*g + U^*g= g - UU^*g,$$
 we obtain the representation
\begin{equation} \label{onedim}
 f = h + (U-I)h_1,
 \end{equation}
where
\begin{equation} \label{formula1}
 h=g - UU^*g, \, h_1=U^*g.
\end{equation}

We observe that the summands of the right hand side of  \eqref{onedim} give rise to the
stationary sequences  $(U^nh)_{n \ge 0}$ and  $(U^{n+1}h_1 - U^n h_1)_{n \ge 0}$  of the
reversed martingale differences and the coboundaries, respectively.
Representation \eqref{onedim} is the basis for applying the martingale ap\-proxi\-ma\-tion method  for
proving the Central Limit Theorem and other pro\-ba\-bi\-lis\-tic limit results. Also certain conditions
 for solvability of the equation \eqref{poisson1}  are known which are based on the statistical ergodic theorem
 for the operator $U^*.$
\begin{remark}
There exist  expressions in terms of the solution of the Poisson equation for the
conditional and the unconditional variances of the martingale difference appearing in
\eqref{formula1}. Indeed, taking into account the first of relations \eqref{formula1},
we obtain (cp. \cite{GoLi1978,Ma1978})
\begin{equation} \label{condvar1}
E^{\, T^{-1}\mathcal F} |h|^2 =
 UU^*|g -UU^*g|^2
= UU^*|g|^2-U|U^* g|^2,
\end{equation}
or
\begin{equation} \label{condvar2}
E^{\, T^{-1}\mathcal F} |h|^2 =
E^{\, T^{-1}\mathcal F} |g|^2- |E^{\, T^{-1}\mathcal F} g|^2.
  \end{equation}
It follows from \eqref{condvar1} that
\begin{equation} \label{var}
  E |h|^2= E|g|^2 - E|U^*g|^2.
  \end{equation}
The latter quantity equals the limiting variance in the Central Limit Theorem for
the sequence $\xi.$
  \qed
\end{remark}

In the present paper a multivariate  analogue of the above si\-tu\-a\-tion is considered.
Some conditions are investigated which ensure the validity of a representation and
relations  similar to \eqref{onedim} and \eqref{formula1}. The unicity issue of
such a representation is also examined. However,
we do not touch applications to limit theorems. Though the case of square-integrable
variables is of main interest, our considerations concern the $L_p$ spaces where
$1\! \le p \! \le \! \infty $ or $ 1\! \le p \! <\! \infty $.
 A multivariate  generalization of the representation  \eqref{onedim} is presented in
 Proposition \ref{martingale+coboundary}. The main assumption here is the solvability of the
equation \eqref{poisson_1}, a higher ana\-lo\-gue of the Poisson equation \eqref{poisson1}. Sol\-va\-bi\-li\-ty
conditions for the equation \eqref{poisson_1} are given in Propositions \ref{convergence1} and
\ref{convergence2}.
A discussion of the definition of multivariate  martingale differences used in the present paper and
comments on the structure of the re\-pre\-sent\-a\-tion  \eqref{representation} and its role
in the investigation of the asymptotics of sums over the random field
can be found in Remarks \ref{condind} and \ref{comment}, respectively.\\
Applications to limit theorems will be presented in separate pub\-li\-ca\-ti\-ons
which are in preparation. In one of them, by M. Weber and the author
 \cite{GoWe2008}, a particular form of the re\-pre\-sen\-ta\-tion
from the present paper is applied to a problem considered in  \cite{FuPe2001} and
related to the so-called Baker sequences. Application of the  martingale
approach allows us to give a complete analysis of possible de\-ge\-ne\-ra\-tions
of the limit in this problem. The second paper, joint with H. Dehling and
M. Denker \cite{DeDeGo2008}, introduces a concept of $U-$ and $V-$statistics
of a measure preserving trans\-for\-ma\-tion and treats asymptotic results for them
 by means of a formalism parallel to that of the present paper; however, it
 is applied to some functional spaces, distinct from the $L_p$ spaces and chosen in
 accordance with the situation considered there.

\section{Notation and statements of results} \label{1}

 Let $T_1, \dots, T_d$ be commuting measure preserving transformations of a
 probability space $(X,\mathcal F, P).$ Denote by  $\Z^d_+$ the additive semigroup
 of $d-$dimensional coordinate vectors with non-negative integral entries. Then the relation  $\mathbf n\! = \!(n_1,
\dots,n_d) \mapsto T^{\mathbf n}=T_1^{n_1} \cdots T_d^{n_d},$ $ \mathbf n
\in \Z^d_+,$ defines a measure preserving action of the semigroup $\Z^d_+$
 on the space $(X,\mathcal F, P).$

Let $\mathcal S_d$ ($\mathcal S_{r,d}$) be the set of all subsets (correspondingly,
of all sets of cardinality  $ r \in [0,d]\,$) of the set  $\N(d)=\{1, \dots, d\}.$
Define for every $\mathit{S} \in \mathcal S_d$ a subsemigroup $\Z^{d,\mathit{S}}_+ \subseteq \Z^d_+$
by the relation $$\Z^{d,\mathit{S}}_+ =\{(n_1,\dots,n_d) \in \Z^d_+:n_k = 0\, \, \text{for all}\,  k \notin \mathit{S}\}.$$
For every  $p \in [1, \infty ]$ let $q=q(p)=p/(1-p)\in [1, \infty].$ Set
$$ U_k f=f\circ T_k$$
for every $f \in L_p= L_p(X,\mathcal F, P)$ and $k \in \N(d).$
 For every $p \in [1, \infty)$ and $k \in \N(d)$ let $ U_k^* $  denote
the conjugate of the operator $U_k.$ The operator $ U_k^* $ is acting on $L_q, 1<q \le \infty.$ The
same symbol $ U_k^* $ denotes an operator on $L_1$ such that
its conjugate is $U_k:L_{\infty} \to L_{\infty}$ (the existence of such an operator follows easily
from the measure-preserving character of  $T_k$).
Every operator $U_k$ is acting on every space $L_p$ as an isometry which preserves values of constant functions
and the cone of nonnegative functions.  Therefore,
$U_k^*$ is acting on every such space as a contraction which preserves
nonnegativity and values of constants.
Furthermore, as was noticed in Sect. 1, for every $k \in \N(d)$ è $n \ge 0$
we have
\begin{equation} \label{relations}
U_k^{*n} U_k^{n}=I,
U_k^{n} U_k^{*n}= E^{T_k^{-n}\mathcal F}.
\end{equation}
If for every  $ i,j \in \N(d),$ $ i \neq j,$ we also have
$$ U_iU_j^*=U_j^*U_i, $$
then the transformations $T_1, \dots, T_d$ are said to be  \linebreak
{\em com\-ple\-te\-ly} {\em com\-muting}.  This property, unlike commutativity, depends on the
probability measure $P.$ It implies that the conditional ex\-pec\-ta\-tions
$$\bigl(E^{T_k^{-n}\mathcal F}\bigr)_{n \ge 0, \,
k \in \N(d)}\,\,$$
mutually commute as well.
Let us set for $n \ge 0$ and  $k \in \N(d)$
$$ \mathcal F_k^n=T_k^{-n} \mathcal F, E_k^n=E^{\mathcal F_k^n},$$
and
 $$\mathcal F_k^{\infty}= \cap_{n=0}^{\infty}\mathcal F_k^n, E_k^{\infty}=E^{\mathcal F_k^{\infty}}. $$
The above commutativity of conditional expectations extends, by passing to the limit, to the family
 $$\bigl(E^n_k\bigr)_{0 \le n \le \infty,\,
k=1,\dots,d}\,\,.$$
 Further, let $\overline{\Z^d_+} $ be a completion of $ \Z^d_+ $,  whose elements $\mathbf n \! = \!(n_1,
\dots,n_d)$ have entries $n_1,
\dots,n_d$  with possible values $0,1, \dots, \infty$. Endow $\overline{\Z^d_+} $ with a natural partial order
extending that of $ \Z^d_+ .$
For every $\mathbf n \! = \!(n_1,
\dots,n_d) \in \overline{\Z^d_+}$  we set
 $$\F^{\mathbf n}= \bigcap_{k=1}^d \F_k^{n_k}, \  E^{\mathbf n}=E^{\F^{\mathbf n}}, $$
 and obtain
$$ E^{\mathbf n}= \prod_{k=1}^d E^{n_k}_k .$$
\begin{remark} \label{condind} Let for $\mathbf m \! = \!(m_1,
\dots,m_d)$, $ \mathbf n \! = \!(n_1,
\dots,n_d),$ $\mathbf m, \mathbf n \in  \overline{\Z^d_+},$
 the relation $\mathbf m \le  \mathbf n $ means, by definition, that
 $m_1 \le n_1,\dots,$ $ m_d \le n_d.$ It is clear that $\F^{\mathbf n} \subseteq \F^{\mathbf m}$
 whenever $\mathbf m \le  \mathbf n .$ Thus, $$\bigl( \F^{\mathbf n}\bigr)_{\mathbf n \in \overline{\Z^d_+}}$$
 is a decreasing filtration parametrized by a partially ordered set $\overline{\Z^d_+}.$ Let $\bigvee $ be the binary operation of taking the coordinatewise maximum in $\overline{\Z^d_+}$.  The commutation relations
 between conditional expectations observed above have the following probabilistic meaning: assume that
 $\mathbf l, \mathbf m, \mathbf n \in  \overline{\Z^d_+}$ and
 $ \mathbf n =  \mathbf l \bigvee \mathbf m,$
 then the  $\s-$fields
 $\F^{\mathbf l}$ è $\F^{\mathbf m}$ are conditionally independent given $\F^{\mathbf n}.$
 Such a property of a filtration (rather for the increasing case than for the decreasing one as in our setup)
 is well-known in the literature (see, for example, \cite{CaWa1975}).   \\
 We will discuss now the definition of reversed martingale differences we choose in this paper.
 We are led to this definition by Proposition \ref{martingale+coboundary} below. A family
$\bigl(\xi_{\mathbf n},
 \F^{\mathbf n}\bigr)_{\mathbf n \in \Z^d_+}$ of random variables defined on $(X,\mathcal F, P),$ and
 sub-$\s$-fields of $\mathcal F,$ is said to be a family of \textit{reversed martingale differences} if
 we have
 \begin{enumerate}
 \item{ for every $\mathbf n \in \Z^d_+$ the random variable $\xi_{\mathbf n}$ is measurable with respect to
 $\F^{\mathbf n}$;}
  \item{ $E^{ \F^{\mathbf m}} \xi_{\mathbf n} =0$ whenever $\mathbf m \nleq \mathbf n .$}
 \end{enumerate}
 This definition without changes applies to any partially ordered set instead of $\Z^d_+.$ Like the
 above conditional independence assumption, it also can be found in the literature. Indeed, in the paper
 \cite{CaWa1975}, which is devoted to stochastic integrals and martingales in ${\mathbb R}^2,$ the concepts of 
 $1$- and $2$-martingales, among several others, are introduced. In the case $d=2$ the definition given above is an
 analogue (for discrete and reversed ''time'') of the property of a random field to be a $1$- and a $2$-martingale
 simultaneously. Comparing the requirements imposed by the definition given above, we see, for example, that it is
 less restrictive than the one given in \cite{BaDo1979}, and more restrictive than the definition in â \cite{Leone1977}.
 Conditions imposed on the filtration is a separate question. As was noticed above, in the setup of the present paper a
 rather special  property of conditional independence holds.
\qed
\end{remark}
{\bf From now on we assume in this paper that the trans\-for\-ma\-tions $T_1, \dots, T_d$
are completely commuting.} \\

For $\mathbf n = (n_1, \dots, n_d) \in \Z^d_+$ we set
$$U^{\mathbf n}= U_1^{n_1} \cdots U_d^{n_d}, \, \, U^{* \mathbf n}=(U^{\mathbf n})^*.$$
For every $ \mathit{S} \in \mathcal S_d$ denote by $\I_\mathit{S}$ the $\s-$field
 of those $A \in \F$ for which the relation $T_k^{-1} A= A$ holds for every $k \in \mathit{S},$ and let
 $E^{\I_\mathit{S}}$ be the corresponding conditional expectation. Notice that for the empty set
 $\emptyset$  $\I_{\emptyset}=\F$ and $E^{\I_{\emptyset}}=I.$  Write $\I_k$
 instead of $\I_{\{k\}}$ for $ k \in \N(d).$
Set $\mathbf {1}_d= (1, \dots,1) \in \Z_+^d$.
 For $\mathbf N=(N_1,\dots,N_d)\in \Z_+^d$ set
$$ S_{\mathbf N} =\! \sum_{0 \le \mathbf n \le \mathbf N - {\mathbf 1}_d} U^{\mathbf n}, \, S_{\mathbf N}^* =\!\sum_{0 \le \mathbf n \le \mathbf N - {\mathbf 1}_d} U^{*\mathbf n},$$
if $\mathbf {1}_d \le \mathbf N,$ and $$ S_{\mathbf N} =0,\, S_{\mathbf N}^* =0 $$ otherwise.

 The following assertion presents a multivariate  analogue of the re\-pre\-sen\-ta\-tion in the form of a
 sum of a martingale difference and a coboundary which was discussed in Section 1. Comments on this
 multi\-va\-ri\-ate representation are given  in Remark \ref{comment} below.
\begin{proposition} \label{martingale+coboundary} Let  $1\le p \le \infty,$ and let for a function
$f \in L_p$ a function $g \in L_p$ solves the  equation
\begin{equation} \label{poisson_1}
f=\bigl(\prod_{k=1}^d (I-U^*_k)\bigr) g.
\end{equation}
 Then $f$ can be represented in the form
\begin{equation}\label{representation}
 f=\sum_{\mathit{S} \in \mathcal S_d} \bigl(\prod_{k \in
\mathit{S}}(U_k-I)\prod_{l \notin
\mathit{S}}(I-U_lU^*_l)\bigr)h_\mathit{S},
\end{equation}
 where for every $\mathit{S} \in \mathcal S_d$ the function $h_{\mathit{S}} \in L_p$  is
  defined by the relation
\begin{equation}\label{formula2}
h_{\mathit{S}}=\bigl(\prod_{m \in \mathit{S}}U^*_m\bigr) g \,.
\end{equation}
Conversely, if for some $g \in L_p$ a function $f \in L_p$ admits the re\-pre\-sen\-ta\-tion \eqref{representation}
with functions $h_{\mathit{S}}$ defined by the relations \eqref{formula2}, then $g$ is a solution of the equation  \eqref{poisson_1}.\\
Let a function $f \in L_p$ admits two representations of the form  \eqref{representation} with the
function
$(h_{\mathit{S}})_{\mathit{S} \in \mathcal S_d} $ and $(h'_{\mathit{S}})_{\mathit{S} \in \mathcal S_d} ,$
correspondingly (now it is not  a priori assumed that any relations of the type of \eqref{formula2} hold). Then for every
 $\mathit{S} \in \mathcal S_d$
$$ \bigl(\prod_{k \in \mathit{S}}(U_k-I)\prod_{l \notin \mathit{S}}(I-U_lU^*_l)\bigr)h'_{\mathit{S}}=\bigl(\prod_{k \in \mathit{S}}
(U_k-I)\prod_{l \notin \mathit{S}}(I-U_lU^*_l)\bigr)h_{\mathit{S}}. $$
\end{proposition}
\begin{remark} \label{comment}
It is clarified here the meaning of decomposition \eqref{representation} and the role of its components in the
asymptotics  of the sums
\begin{equation} \label{sum}
S_{\mathbf N} f=\! \sum_{0 \le \mathbf n \le \mathbf N- {\mathbf 1}_d} U^{\mathbf n}f.
\end{equation}
Let $\mathit{S} \in \mathcal S_{r,d}$. The summand
$$ A_{\mathit{S}}= \bigl(\prod_{k \in
\mathit{S}}(U_k-I)\prod_{l \notin
\mathit{S}}(I-U_lU^*_l)\bigr)h_{\mathit{S}}  $$
of the right-hand side of relation \eqref{representation} satisfies the equations
\begin{equation}
E^1_t A_{\mathit{S}}=0, t \notin \mathit{S}.
\end{equation}
To establish this fact, we represent $ A_{\mathit{S}} ,$ using the commutation relations, as $ A_{\mathit{S}}= \bigl(\prod_{l \notin
\mathit{S}}(I-U_lU^*_l)\bigr)B_{\mathit{S}}$ and then apply the relations $U_tU^*_t(I-U_tU^*_t)=0,$ $t \notin \mathit{S}.$ This implies 
that for every $\mathbf m \in \Z_+^d$ the fa\-mi\-ly 
$ \bigl(U^{\mathbf m+\mathbf n} A_{\mathit{S},}\mathcal F^{\mathbf m+\mathbf n}\bigr)_{\mathbf n \in \Z_+^{d,\mathbb N(d) \setminus \mathit{S}}}$ is a
$(d\!-\!r)-$dimensional random field of reversed martingale differences (in the sense of Remark \ref{condind}.)
The latter means that
for $\mathbf n \in \Z_+^{d,\mathbb N(d) \setminus \mathit{S}}$ the random variable  $U^{\mathbf m+\mathbf n} A_{\mathit{S}}$
is $\mathcal F^{\mathbf m+\mathbf n}$-me\-a\-sur\-able and satisfies the relations
$$E^{\mathbf m+\mathbf n +\mathbf e_l} U^{\mathbf m+\mathbf n} A_{\mathit{S}}=0,  l \notin \mathit{S}.  $$
Here we set $ \mathbf e_l=(\d_{l,1}, \dots, \d_{l,d}),$ where $\d_{l,m}=1$ for $l=m,$ and $\d_{l,m}=0,$ if $l\neq m.$ \\
Further, for every $\mathbf m \in \Z_+^d$ the family $ (U^{\mathbf m+\mathbf n} A_{\mathit{S}})_{\mathbf n \in \Z_+^{d, \mathit{S}}}$ is a
$r-$dimensional stationary random field of $r-$coboundaries, that is it has the form \linebreak
$ \bigl(U^{\mathbf n}\bigl(\prod_{k \in
\mathit{S}}(U_k-I)\bigr) C_{\mathit{S},\mathbf m})\bigr)_{\mathbf n \in \Z_+^{d, \mathit{S}}}.$  In particular, this
implies that the sums
$$ \sum_{\{\mathbf n =(n_1, \dots,n_d) \in \Z_+^{d, \mathit{S}}:\, n_k \in [0, N_k-1]\, \text{for every}\, k \in \mathit{S}\}} U^{\mathbf m+\mathbf n} A_{\mathit{S}} $$
are bounded in $L_p.$   Putting off  the analysis of distributions of the sums \eqref{sum} to another case, we will describe
now at the heuristic level the role played by decomposition \eqref{representation} in this issue. We have, in view of  \eqref{representation},
\begin{equation} \label{sum1}
\begin{split}
S_{\mathbf N} f&=S_{\mathbf N}\Bigl(\sum_{\mathit{S} \in \mathcal S_d} \bigl(\prod_{k \in
\mathit{S}}(U_k-I)\prod_{l \notin
\mathit{S}}(I-U_lU^*_l)\bigr)h_{\mathit{S}}\Bigr) \\
&=S_{\mathbf N}\sum_{\mathit{S} \in \mathcal S_d} A_{\mathit{S}}
= \sum_{\mathit{S} \in \mathcal S_d} S_{\mathbf N} A_{\mathit{S}}.
\end{split}
\end{equation}
Keeping  $\mathit{S}$ fixed, the behavior of the sums $S_{\mathbf N} A_{\mathit{S}}$ by $\mathbf N=(N_1, \dots, N_d) \to \infty$
depends on existence of the moments and some other properties of the summands. For $p \ge 2$ the above-mentioned cobounadry
properties of $A_{\mathit{S}}$ guarantee the boundedeness of $L_2$-norms of the random variables
 \begin{equation} \label{bound}
 \bigl(\prod_{l \notin \mathit{S}} N_l\bigr)^{-1/2}S_{\mathbf N} A_{\mathit{S}}.
 \end{equation}
 If $(T^{\mathbf n})_{\mathbf n \in \Z_+^d}$  is a mixing action and $\bigr(\prod_{l \notin
\mathit{S}}(I-U_lU^*_l)\bigr)h_{\mathit{S}}\neq 0$ (we will call such a function  $f$ $\mathit{S}$-{\em nondegenerate}), the $L_2$-norms of such variables have a (finite) positive limit. Moreover, notice (though we do not need it
in the present paper) that these variables converge in distribution to a centered Gaussian law whose
variance is the square  of this limit.  In case of $\emptyset-$non-degeneracy of $f,$ the summand $$ S_{\mathbf N} A_{\emptyset}=S_{\mathbf N} \bigl(\prod_{l =1}^d(I-U_lU^*_l)\bigr)g $$ dominates in the sum $\sum_{\mathit{S} \in \mathcal S_d}S_{\mathbf N} A_{\mathit{S}}.$
 Indeed, denoting by $|\cdot|_2$  the $L_2$-norm, we obtain
\begin{equation} \label{variance}
 \s^2_{\emptyset}(f)= \bigl|\bigl(\prod_{l =1}^d(I-U_lU^*_l)\bigr)g \bigr|^2_2 \, \, \Bigl(=\sum_{r=0}^d (-1)^r\sum_{\mathit{S} \in \mathcal S^{r,d}}
 \bigl|\prod_{k \in \mathit{S}} U^*_k g \bigr|_2^2\Bigr).
 \end{equation}
 Since the reversed martingale differences
 $\bigl( U^{\mathbf n}\bigl(\prod_{l =1}^d(I-U_lU^*_l)\bigr)g \bigr)_{\mathbf n \in \Z_+^d}$
 are mutually orthogonal, we have for $\mathbf N=(N_1, \dots, N_d)$
 $$\bigl| S_{\mathbf N} A_{\emptyset}\bigr|^2_2= \bigl(\prod_{k=1}^d N_k\bigr) \s^2_{\emptyset}(f).$$
 Comparing this amount with \eqref{bound} for $\mathit{S} \neq \emptyset$, it is clear that $S_{\mathbf N} A_{\emptyset}$ do\-mi\-na\-tes
 in the sums $S_{\mathbf N}f$ as $\mathbf N \to \infty$ whenever $ \s^2_{\emptyset}(f )> 0$.
  This fact is crucial when one proves limit theorems for sums  $S_{\mathbf N} f$ by reduction the problem to the case of
  reversed martingale differences. It also shows that  $ \s^2_{\emptyset}(f )$ does not depend on the choice of
  the representation of the type of  \eqref{representation}, and that the notation introduced above is consistent. Moreover, it is
  clear that the random variable $A_{\emptyset}$ generating a $d-$di\-men\-si\-o\-nal field
 of reversed martingale differences is uniquely de\-ter\-min\-ed. This analysis of the asymptotics can be
 continued to obtain the uniqueness of summands in the representation of the type \eqref{representation} on the way
 distinct from that taken in the proof of Proposition  \ref{martingale+coboundary}.
 \qed
\end{remark}
In the rest of the present section conditions for solvability of the equation
 \eqref{poisson_1} are discussed, and a description of the set of its solutions is given.

 The following remark will be needed in the course of the proof of the
 Proposition \ref{convergence1} to identify the limit in the statistical ergodic
 theorem for the operators  $U_k^*.$
 \begin{remark} \label{ergod} Here some general properties  of the actions
 under con\-si\-dera\-tion are summarized. Since the transformations
 $T_1, \dots, T_d$ commute, the conditional expectations  $E^{\I_\mathit{S}}, \mathit{S} \in  \mathcal S_d,$
 commute as well. Notice that for $k \in \N(d)$ \begin{equation} \label{ergincl}
 \I_k \subseteq \F^{\infty}_k.
 \end{equation}
 Let us make clear the interrelation between the invariant elements of the
 operators $U_k$ and $U^*_k$. Assume that for some  $k \in \N(d)$ a certain $ f \in L_p$ satisfies
$U_k f = f. $ Apply  $U^*_k$ to the both parts of the last equation. Then the relation  $U^*_kU_k=I$
yields  $ U^*_kf=f.$
Conversely, let $U_k^* f = f. $ Then for all $n \ge 0$  $U^{*n}_k f = f ,$
$U_k^nU^{*n}_k f = U_k^n f $ and, consequently, $E^n_kf= U_k^n f.$ Since the operator
$ U_k$ is an isometry, the conditional expectation in the left-hand side preserves the $L_p$-norm
of $f,$ which is only possible if the expectation acts on $f$ identically. The latter
means that $f$ is $\mathcal F_k^n-$measurable. Since $n$ is arbitrary, $f$ is
$ \mathcal F_k^{\infty}-$measurable. Further, it follows from the relations between
$U_k$ and $U^*_k$ that they act on the space of $ \mathcal F_k^{\infty}$-measurable  $L_p-$functions
as mutually inverse isometries which implies $U_k f = f .$
Therefore, the operators $U_k$ and $U_k^*$ have the same invariant elements in the spaces  $L_p.$
The same conclusion also holds for every $\mathit{S} \in \mathcal S_d$ for jointly invariant elements
of every of two sets of operators:  $\{U_k:k \in \mathit{S}\}$ and $\{U^*_k: k \in \mathit{S}\}.$
Hence, we have for every $\mathit{S} \in \mathcal S_d$ and every
$p \in [1, \infty]$ that
$$\{f\!\in L_p\!:\!U^*_kf=f, k \!\in \mathit{S}  \}\!=\{f\!\in L_p\!:\!U_kf=f, k \!\in \mathit{S}  \}=$$
$$ =  \{f \! \! \in L_p\!:\! f
\,\text{is measurable with respect to}\, \I_\mathit{S}  \}. $$
\qed
\end{remark}
Let us call a function $g \in L_p$
{\em normal}, if
$$ g= \Bigl(\!\prod_{\, k \in \N(d)}\! (I- E^{\I_k})\Bigr)\,\,g.$$
Normality of $g \in L_p$ is equivalent to
$$ E^{\I_k}g=0, k \in \N(d). $$
Denote by $\Ker (A)$ and $\Ran (A)$ the kernel and the image of a linear operator $A$, respectively .
\begin{proposition} \label{convergence1}
Let  $ p \in [1, \infty].$
\begin{enumerate}
\item Every $f \in L_p$, for which equation \eqref{poisson_1} has a solution $g \in L^p $,
is normal.
\item A function $g \in L_p$ is a solution of equation \eqref{poisson_1} if and only if it can be
represented in the form $g=g'+e,$ where $g'$ is a normal solution of equation \eqref{poisson_1} and $ e \in
\Ker\bigl( \prod_{\, k \in \N(d)}\! (I- E^{\I_k})\bigr).$
Equation \eqref{poisson_1} has at most one normal solution.
\item Let $p <\infty$ and $f \in L_p$ is a normal function. Equation
\eqref{poisson_1} has a solution in $L^p $ if and only if
the limit
\begin{equation} \label{series1}
\underset{\mathbf N=(N_1,\dots,N_d) \to \infty}{\lim}(N_1 \cdots N_d)^{-1}\!\sum_{0 \le \mathbf M \le \mathbf N - {\mathbf 1}_d }
\,  S_{\mathbf M}^* f
\end{equation}
exists in the $L_p$-norm. This limit represents a normal solution of
equation \eqref{poisson_1}.
\end{enumerate}
\end{proposition}
Substituting the normality assumption by a stronger condition, one can simplify the
solvability criterion of \eqref{poisson_1} and the procedure of constructing its
solution. Let us call a function $f \in L_1$ {\em strictly normal}
whenever
 $ f = \prod_{k \in \N_d}(I-E^{\infty}_k) f$ (an equivalent property says that for all $k \in \N(d)$ $E^{\infty}_kf=0.$)
The strict normality is stronger than the normality because
 $\I_k \subset \F^{\infty}_k$ (recall that  $E^{\infty}_k= E^{\F^{\infty}_k}$).

 The strict normality of $f \in L_p$ can be characterized by any of the following
 properties (where convergence is assumed in the sense of the $L_p-$norm):
 \begin{equation}
 E^{(n_1, \dots, n_d)} f \underset{\max(n_1, \dots, n_d) \to \, \infty}{\rightarrow} 0,
 \end{equation}
 or
 \begin{equation}
  \label{strict} U^{*(n_1, \dots, n_d)} f \underset{ \max(n_1, \dots, n_d) \to \, \infty}{\to} 0.
  \end{equation}
\begin{proposition} \label{convergence2}
Let $ p \in [1, \infty).$
\begin{enumerate}
\item If a function $f \in L_p$
can be represented as
\eqref{poisson_1}, where $ g \in  L_p$ is strictly normal, then $f$
is strictly normal, and $g$ can be represented in the form
\begin{equation} \label{series}
g = \sum_{\mathbf n \in \Z_+^d} U^{*\mathbf n} f\, \, \Bigl(\,\overset{\text{def}}{=}\!\underset{\mathbf N=(N_1,\dots,N_d) \to \infty}{\lim}
S^*_{\mathbf N}f\Bigr),
\end{equation}
where the series converges in the $L_p$-norm.
 Equation \eqref{poisson_1} has at most one strictly normal
 solution in $L_p$.
 \item
 Conversely, let for some strictly normal function $f \in L_p$ series \eqref{series}
 converges in the $L_p$-norm. Then the sum of series \eqref{series} presents a strictly
 normal solution of equation \eqref{poisson_1}.
\item For every strictly normal function $f \in L_p$  it follows from the convergence of series \eqref{series}
that for every  $ \mathit{S} \in \mathcal \mathit{S}_d$ the series
 \begin{equation} \label{seriesS}
\sum_{\mathbf n \in \Z_+^{d,\mathit{S}}} U^{*\mathbf n} f\, \, \Bigl(\,\overset{\text{def}}{=}\!\underset{N_k \to \infty,\, k \in \mathit{S}}{\lim} \sum_{\mathbf n \in \Z_+^{d,\mathit{S}}} U^{*\mathbf n} f\Bigr)
\end{equation}
converges in the $L_p-norm$.
\end{enumerate}

\end{proposition}
\begin{remark} For $d \ge 2$ the convergence of series \eqref{series} seemingly does not imply that \eqref{strict} holds (that is that $f$ is strictly normal). However, one can omit in assertions (2) and
(3) of Proposition \ref{convergence2}  the assumption that $f$ is strictly normal, if one assumes instead, in addition to the convergence of  \eqref{series}, the convergence of
 \eqref{seriesS} for every set $\mathit{S}$ of cardinality 1.
\qed
\end{remark}
\begin{remark} Propositions  \ref{convergence1} and  \ref{convergence2} can be extended to the case of
the space $L_{\infty}$, if one considers the convergence of the series in the $L_1$-topology of
the space $L_{\infty}$ instead of their convergence in the $L_{\infty}$-norm. Further, for $1<p< \infty$ the requirement
of the existence of the limit \eqref{series1} in Proposition \ref{convergence1} as a sufficient condition
of solvability  of equation \eqref{poisson_1} can be weakened to that of boundedness in $L_p$ of the corresponding
sequence of partial sums.
\end{remark}
\begin{ex}
 Let $d=2.$  If for a strictly normal function $f \in L_p$  the series
$\sum_{n_1, n_2 =0}^{\infty} U_1^{*n_1}U_2^{*n_2}f $ converges in the
$L_p$-norm, then $f$ admits a representation in the form
 $$f=
C_{\emptyset} +(U_1-I)C_1 +(U_2-I)C_2+(U_1-I)(U_2-I)C_{1,2}, $$ where
the functions $C_{\emptyset},C_1, C_2, C_{1,2} \in L_p$ are strictly normal and
$$E^{T_i^{-1}\mathcal F}C_{\emptyset}=0 \,
(i=1,2),$$ $$ E^{T_2^{-1}\mathcal F}C_1=0,E^{T_1^{-1}\mathcal
F}C_2=0.$$
If the transformations  $T_1$ and  $T_2$ are ergodic, then the strictly normal functions
$C_{\emptyset},C_1, C_2, C_{1,2}$ are uniquely determined.
\qed
\end{ex}

\section{Proofs} \label{2}

In the course of proofs of Propositions  \ref{martingale+coboundary},\,\ref{convergence1} and \ref{convergence2} the following assertion will be needed.
\begin{lemma} \label{projection} For every $\mathit{S} \in \mathcal S_d$ the following relations hold:
\begin{equation} \label{relation1}
\begin{split}
&\Bigl(\prod_{k \in \mathit{S}} (I- E^{\I_k})\Bigr)\!\Bigl( \prod_{k \in \mathit{S}}\! (I-U^*_k)\Bigr)\\
= \Bigl(\prod_{k \in \mathit{S}}\! &(I-U^*_k)\Bigr) \Bigl( \prod_{k \in \mathit{S}} (I- E^{\I_k})\Bigr)
=\prod_{k \in \mathit{S}}\! (I-U^*_k),
\end{split}
\end{equation}
\begin{equation} \label{relation2}
\begin{split}
&\Bigl(I \!- \! \!\prod_{k \in \mathit{S}} (I- E^{\I_k})\Bigr)\!  \Bigl(\prod_{k \in \mathit{S}}\! (I-U^*_k)\Bigr)\\
=& \Bigl(\! \prod_{k \in \mathit{S}} \!  (I-U^*_k)\Bigr)\, \Bigr(I-\prod_{k \in \mathit{S}} (I- E^{\I_k})\Bigr)
 = 0,
\end{split}
\end{equation}
\begin{equation} \label{relation3}
\begin{split}
\Ker\Bigl(\prod_{k \in \mathit{S}} (I-U^*_k)\Bigr) = \Ker\Bigl(\prod_{k \in \mathit{S}} (I-U_k)\Bigr)= \Ker \Bigl( \prod_{k \in \mathit{S}} (I- E^{\I_k})\Bigr).
\end{split}
\end{equation}
\end{lemma}
\begin{proof}[Proof of Lemma 1]
For every $k\in $
$ \N(d)$ we have $E^{\I_k}U^*_k $ $= E^{\I_k}$
(this follows, for example, from the obvious identity $U_k E^{\I_k} = E^{\I_k}$
applied to the dual space).
This implies $(I-E^{\I_k})(I-U^*_k)= I-U^*_k.$
Taking the product of these relations over all $k \in \mathit{S}$
(the order of the multipliers, in view of their commutativity, is of no importance here) gives \eqref{relation1}.
Subtracting the both parts of  \eqref{relation1}  from  the operator $\prod_{k \in \mathit{S}}\! (I-U^*_k)$,
 \eqref{relation2} follows. \\
Let us prove \eqref{relation3}.
According to Remark \ref{ergod} for all $k \in \N(d)$ the relations
$ \Ker (I-U^*_k)=\Ker (I-U_{k})=\Ker (I- E^{\I_{k}})$ hold. Equalities \eqref{relation3} are consequences of
these relations for $k \in \mathit{S}$ and the fact that the kernel of the product of two commuting bounded operators
is the closure of the sum of their kernels.
\end{proof}
\begin{proof}[Proof of Proposition 1]
Since for $g$ holds \eqref{poisson_1}, we have
\begin{equation} \label{martcob}
\begin{split}
f =\bigl(\prod_{k=1}^d (I-U^*_k)\bigr) g = \bigl(\prod_{k=1}^d [(I-U_kU^*_k) +(U_k-I)U^*_k)]\bigr) g \\
=\sum_{\mathit{S} \in {\mathcal S}_d} \bigl(\prod_{k \in \mathit{S}}(U_k - I)\prod_{l \notin \mathit{S}}(I-U_lU^*_l)\prod_{m \in \mathit{S}}U^*_m\bigr)g\\
=\sum_{\mathit{S} \in {\mathcal S}_d} \bigl(\prod_{k \in \mathit{S}}(U_k-I)\prod_{l \notin \mathit{S}}(I-U_lU^*_l)\bigr)h_{\mathit{S}},
\end{split}
\end{equation}
where
$ h_{\mathit{S}} = \bigl(\prod_{m \in \mathit{S}}U^*_m\bigr)g.$ \\
Reversing this chain of equalities, we see that it follows from \eqref{representation} and \eqref{formula2} that
$g$ satisfies the relation \eqref{poisson_1}.\\
 Let us prove now the assertion about the uniqueness of the representation  \eqref{representation}.
 We set $H_{\mathit{S}}=h'_{\mathit{S}}-h_{\mathit{S}}.$ Subtracting one of the representations
 from another, the relation
\begin{equation} \label{unicity1}
\sum_{\mathit{S} \in {\mathcal S}_d} \bigl(\prod_{k \in \mathit{S}}(U_k-I)\prod_{l \notin \mathit{S}}(I-U_lU^*_l)\bigr)H_{\mathit{S}}=0,
\end{equation}
is obtained. It suffices now to deduce from \eqref{unicity1} that for every $\mathit{S} \in {\mathcal S}_d$
\begin{equation} \label{unicity2}
 \bigl(\prod_{k \in \mathit{S}}(U_k-I)\prod_{l \notin \mathit{S}}(I-U_lU^*_l)\bigr){H_\mathit{S}}=0.
 \end{equation}
At the initial step (we give number zero to it) we apply the operator $\prod_{k \in \Z_d}U^*_k$ to the both parts of equation \eqref{unicity1},
keeping in mind that operators with different indices commute while for operators
with the same index
the relations  $ U^*_l (I-U_lU^*_l)=0, U^*_l(U_l-I)=I-U^*_l, l \in \Z(d),$ hold. It
is  clear that
the operator $\prod_{k \in \Z_d}U^*_k$ vanishes on all summands in the left-hand side of \eqref{unicity1}, except for that for which  $\mathit{S} = \Z(d).$ This implies
$$ \bigl( \prod_{k \in \Z_d}(I-U^*_k)\bigr)H_{\Z(d)}  =0,$$
which follows, in view of the first of relations \eqref{relation3}, that
$$\bigl(\prod_{k \in \Z_d}(U_k-I)\bigr)H_{\Z(d)}  =0.$$
Therefore, the relation \eqref{unicity2} for $ \mathit{S} = \Z(d)$
 is obtained.
Subtracting this relation "of level $d$" \,from \eqref{unicity1},
the relation
\begin{equation} \label{unicity3}
\sum_{r=0}^{d-1}\sum_{\mathit{S} \in {\mathcal S}_{r,d}} \bigl(\prod_{k \in \mathit{S}}(U_k-I)\prod_{l \notin \mathit{S}}(I-U_lU^*_l)\bigr)H_{\mathit{S}}=0
\end{equation}
is established.
At the step one  each of $d$ products
$$U^*_2\cdots U^*_d, \,U^*_1 U^*_3 \cdots U^*_d,\dots,\, U^*_1 \cdots U^*_{d-1}$$
of $d\!-\!1$ operators is consequently applied to the equation obtained before. This gives the equalities
$$ \bigl(\prod_{k \in \mathit{S}}(I-U^*_k)\prod_{l \notin \mathit{S}}(I-U_lU^*_l)\bigr)H_{\mathit{S}}  =0, \mathit{S} \in \mathcal S_{d-1,d}.$$
Using again the relation \eqref{relation3}, we see that
$$ \bigl(\prod_{k \in \mathit{S}}(U_k-I)\prod_{l \notin \mathit{S}}(I-U_lU^*_l)\bigr)H_{\mathit{S}}  =0, \mathit{S} \in \mathcal S_{d-1,d}.$$
Therefore, we obtained such way every of $d$ equalities of level  $d-1$ from \eqref{unicity2}.
Subtract them from  \eqref{unicity3} and continue the process. At the $r$-th step
all those  $d \choose r$ relations from
 \eqref{unicity2} will be obtained, which correspond to $ \mathit{S} \in \mathcal S_{d-r,d}.$ The process finishes at the  $d$-th
 step by obtaining the relation \eqref{unicity2} for $\mathit{S}=\emptyset.$
\end{proof}
\begin{proof}[Proof of Proposition \ref{convergence1}]
  It follows from relations \eqref{poisson_1} è \eqref{relation1} that
  \[
\begin{split}f = & \Bigl(\prod_{l=1}^d \!(I-U^*_l)\Bigr)g
= \Bigl(\prod_{k \in \N(d)} (I- E^{\I_k})\Bigr)
\Bigl(\prod_{k \in \N(d)} (I-U^*_k)\Bigr)g
\\= &\Bigl(\prod_{k \in \N(d)} (I- E^{\I_k})\Bigr)f ,
\end{split}
\]
which proves (1). \\
Let now $g \in L_p$ be a certain solution of \eqref{poisson_1}. Set $g'= \bigl(\prod_{k \in \N(d)} (I- E^{\I_k})\bigl) g$ and
$e=\bigl(I-\prod_{k \in \N(d)} (I- E^{\I_k})\bigl)g$. It follows from \eqref{relation1} by $\mathit{S}=\N(d)$ that
 $g'$ is a solution of  \eqref{poisson_1}. This solution is normal, since $\prod_{k \in \N(d)} (I- E^{\I_k})g'=g'.$ Let $g_1$ and $g_2$ be two normal solutions of equation
 \eqref{poisson_1}. Then $g_2 - g_1 \in \Ker\bigl(\prod_{k \in \N(d)} (I-U^*_k)\bigr),$ which, combined with
\eqref{relation3}, implies   $g_2 - g_1 \in \Ker (\prod_{k \in \N(d)} (I- E^{\I_k}) ).$ On the other hand, in view
of normality of the solutions  $g_1$ è $g_2$ we have $g_2 - g_1 \in \Ran \bigl(\prod_{k \in \N(d)} (I- E^{\I_k})\bigl).$ This implies $g_2-g_1 =0,$  which proves (2). \\
Let $f$ admit representation \eqref{poisson_1}. In view of just established item (2) of Proposition \ref{convergence1}, the function $g$ in this representation can be (and will be) chosen to be normal. Set $\mathbf M=(M_1,\dots,M_d)$ and $\mathbf N=(N_1,\dots,N_d).$  Then we have
$$ S^*_{\mathbf M} f=\sum_{0 \le \mathbf n \le \mathbf M- {\mathbf 1}_d} U^{*\mathbf n}f =
\bigl(\prod_{k=1}^d \bigl(\sum_{n=0}^{M_k-1}U_k^{*n}\bigr)\bigr) f =\bigl(\prod_{k=1}^d \bigl(I-U_k^{*M_k}\bigr)\bigr) g $$
and
\begin{equation} \label{calcul1}
\begin{split}
(N_1 \cdots N_d)^{-1}\!\!\!\sum_{0 \le \mathbf M \le \mathbf N- {\mathbf 1}_d} S^*_{\mathbf M}  \,  f \,\,& \\
= (N_1 \cdots N_d)^{-1}\sum_{0 \le \mathbf M \le \mathbf N- {\mathbf 1}_d}\!
 \!\! \bigl( \prod_{k=1}^d & \bigl( I-U_k^{*M_k}\bigr)\bigr)\, g \\
=\!(N_1 \cdots N_d)^{-1}\sum_{0 \le \mathbf M \le \mathbf N- {\mathbf 1}_d}  \sum_{r=0}^d & (-1)^r  \sum_{\mathit{S} \in  \mathcal S_{r,d}}
 \bigl( \prod_{l \in \mathit{S}} U_l^{*M_l}\bigr)\, g   \\
= \! \sum_{r=0}^d (-1)^r  \sum_{\mathit{S} \in  \mathcal S_{r,d}}(N_1 \cdots N_d )^{-1}\, \, & \sum_{0 \le \mathbf M \, \le \mathbf N- {\mathbf 1}_d}\! \bigl( \prod_{l \in \mathit{S}}  U_l^{*M_l}\bigr) \, g  \\
= \! \sum_{r=0}^d \, (-1)^r \,\, \sum_{\mathit{S} \in  \mathcal S_{r,d}}\, \Bigl(
\,\, &  \prod_{l \in \mathit{S}}\ \bigl( N_l^{-1} \sum_{M_l= 0}^{N_l - 1}\!  U_l^{*M_l} \bigr)\Bigr) \, g  \\
\underset{\mathbf N \to \infty}{\to} \sum_{r=0}^d (-1)^r \sum_{\mathit{S} \in  \mathcal S_{r,d}} E ^{\I_\mathit{S}}g,
\end{split}
\end{equation}
where on the last stage the multiparameter statistical ergodic theorem was applied.
Since $g$ is normal, $ E ^{\I_{\mathit{S}}}g=0$  for every non-empty $\mathit{S}$, while
$E ^{\I_{\emptyset}}g=g.$\\
Conversely, let for a normal function  $f\in L_p$ in the $L_p$-norm there exists  the limit
\begin{equation} \label{limit}
\underset{\mathbf N =(N_1, \dots,N_d)\to \infty}{\lim}(N_1 \cdots N_d)^{-1}\!\!\sum_{0 \le \mathbf M\le \mathbf N - {\mathbf 1}_d} S^*_{\mathbf M} \, f=g.
\end{equation}
 The operators $ U_1^*, \dots, U_d^*$ send normal functions to normal ones. Hence, $g$ is normal,
as a limit of normal functions. Further, acting as in \eqref{calcul1}, we obtain
\begin{equation} \label{calcul2}
\begin{split}
\bigl(\prod_{k \in \N(d)}\! (I-U^*_k)\bigr)g & \\
=  \underset{\mathbf N =(N_1, \dots,N_d)\to \infty}{\lim}(N_1 \cdots N_d)^{-1}\!\!\sum_{0 \le \mathbf M \le \mathbf N- {\mathbf 1}_d}
& S^*_{\mathbf M}  \bigl(\prod_{k \in \N(d)}\! (I-U^*_k)\bigr)  \, f \\
= \underset{\mathbf N =(N_1, \dots,N_d)\to \infty}{\lim}(N_1 \cdots N_d)^{-1}\!\! \sum_{0 \le \mathbf M \le \mathbf N- {\mathbf 1}_d}& \bigl(\prod_{k=1}^d \bigl(I-U_k^{*M_k}\bigr)\bigr)f \\
= \underset{\mathbf N =(N_1, \dots,N_d)\to \infty}{\lim} \!(N_1 \cdots N_d)^{-1}\sum_{0 \le \mathbf M \le \mathbf N- {\mathbf 1}_d} & \sum_{r=0}^d (-1)^r  \sum_{\mathit{S} \in  \mathcal S_{r,d}}\bigl( \prod_{l \in \mathit{S}} U_l^{*M_l}\bigr)\, f \\
= \! \sum_{r=0}^d (-1)^r \underset{\mathbf N =(N_1, \dots,N_d)\to \infty}{\lim}  \sum_{\mathit{S} \in  \mathcal S_{r,d}} &
\Bigl(\prod_{l \in \mathit{S}}\bigl( N_l^{-1} \sum_{M_l= 0}^{N_l - 1}\!  U_l^{*M_l} \bigr)\Bigr)\, \, f  \\
= \sum_{r=0}^d (-1)^r \sum_{\mathit{S} \in  \mathcal S_{r,d}}& E ^{\I_\mathit{S}}f = f. \\
\end{split}
\end{equation}
\end{proof}
\begin{proof} [Proof of Proposition \ref{convergence2}] The space of strictly normal $L_p-$function is invariant
with respect to the operators
 $U^*_k, k=1, \dots,d.$ This implies that $f$ is strictly normal. Further,
\begin{equation} \label{calcul3}
\begin{split}
&\!\sum_{0 \le \mathbf n \le \mathbf N- {\mathbf 1}_d} U^{*\mathbf n}f \!
=\sum_{0 \le \mathbf n \le \mathbf N- {\mathbf 1}_d} U^{*\mathbf n}\bigl(\prod_{k=1}^d (I-U^*_k)\bigr) g \\
&=\Bigl(\prod_{k=1}^d \bigl((I-U^{*}_k) \sum_{n_k=0}^{N_k-1} U_k^{*\,n_k} \bigr)\Bigr) g=\Bigl(\prod_{k=1}^d \bigl(I-U^{*N_k}_k \bigr)\Bigr)g\\
&=\sum_{r=0}^d (-1)^r \sum_{\mathit{S} \in  \mathcal S_{r,d}}\prod_{k \in \mathit{S}} U^{*N_k}_k g
\underset{\mathbf N \to \infty}{\to} \sum_{r=0}^d (-1)^r \sum_{\mathit{S} \in  \mathcal S_{r,d}}\prod_{k \in \mathit{S}} E_k^{\infty}g =g,
\end{split}
\end{equation}
where the last equation follows from the strict normality of $g$.
It follows from this representation that a normal solution is unique (this
follows also from Proposition \eqref{convergence2}), and (1) is proved. \\
 Starting to prove (2), let $g$ denote the sum of the series \eqref{series}. The strict normality of
 $g$ is a consequence of the strict normality of $f$ and the fact that the subspace of strictly normal $L_p-$functions is closed and invariant with respect to the operators  $U^*_k, k=1, \dots,d$. Further, analogously to \eqref{calcul3},
 \begin{equation} \label{calcul4}
\begin{split}
\Bigl(\prod_{k=1}^d (I-U^*_k) \Bigr) g& \\
= \Bigl(\prod_{k=1}^d (I-U^*_k)\Bigr)\, &\underset{\mathbf N =(N_1, \dots,N_d)\to \infty}{\lim}\sum_{0 \le \mathbf n \le \mathbf N- {\mathbf 1}_d}U^{*\mathbf n} f \\
=\underset{\mathbf N =(N_1, \dots,N_d)\to \infty}{\!\lim}&\Bigl(\prod_{k=1}^d \bigl(I-U^{*N_k}_k \bigr)\Bigr) f\\
= \sum_{r=0}^d (-1)^r \!\sum_{\mathit{S} \in  \mathcal S_{r,d}}& \underset{\mathbf N =(N_1, \dots,N_d)\to \infty}{\lim}\Bigl(\prod_{k \in \mathit{S}} U^{*N_k}_k \Bigr) f\\
&\,\,\,\,=\,\,\,\,f.
\end{split}
\end{equation}
Assertion of item (3) follows from items (1) and (2), if one first uses  (2), then notices that
 $f=\bigl(\prod_{k \in \N(d)}(I-U^{*k})\bigr)g$ implies  $f=\bigl(\prod_{k \in \mathit{S}}(I-U^{*k})\bigr)g'$ with
 $g'=\bigl(\prod_{k \notin \mathit{S}}(I-U^{*k})\bigr)g,$ and, finally, applies item (1) to the semigroup
 $\Z_+^{d,\mathit{S}}.$
\end{proof}
\section{Acknowledgement}
The author is grateful to the staff of the Erwin Schr\"odinger
Institute of Mathematical Physics (ESI) in Vienna  and to the
organizers of the workshops "Algebraic, geometric and
probabilistic aspects of amenability", "Amenability beyond groups"
(2007) and "Structural Probability" (2008) held at ESI, especially
to Klaus Schmidt, Vadim Kaimanovich and Anna Erschler.

\end{document}